\documentclass[a4paper,12pt]{article}
\usepackage{geometry,latexsym,amssymb,amsmath,amsthm,color,bm}
\usepackage[latin5]{inputenc}
\usepackage{enumerate}
\usepackage{enumitem}
\usepackage[T1]{fontenc}
\usepackage{authblk}
\usepackage[bookmarks=true,
	bookmarksnumbered=true, breaklinks=true,
	pdfstartview=FitH, hyperfigures=false,
	plainpages=false, naturalnames=true,
	colorlinks=true,
	pdfpagelabels]{hyperref}
\usepackage[all,v2,cmtip,2cell]{xy}
\UseAllTwocells
\usepackage{palatino}
\usepackage{indentfirst}
\usepackage{titlesec}
\usepackage{graphics}
\usepackage[numbers]{natbib}

\titleformat{\section}[block]
{\fontsize{12}{15}\bfseries\filcenter}
{\thesection.}
{1em}
{}
\titleformat{\subsection}[block]
{\fontsize{12}{15}\bfseries}
{\thesubsection}
{1em}
{}

\theoremstyle{plain}
\newtheorem{theorem}{Theorem}[section]
\newtheorem{proposition}[theorem]{Proposition}
\newtheorem{lemma}[theorem]{Lemma}
\newtheorem{corollary}[theorem]{Corollary}

\theoremstyle{definition}
\newtheorem{definition}[theorem]{Definition}
\newtheorem{example}[theorem]{Example}

\newtheorem{remark}[theorem]{Remark}

\def\Ker{\operatorname{Ker}}
\def\Nat{\operatorname{Nat}}
\def\coker{\operatorname{coker}}
\def\w{\widetilde}

\def\GpGd{\bm{\mathsf{GpGd}}}
\def\AGpGd{\bm{\mathsf{AbGpGd}}}
\def\XMod{\bm{\mathsf{XMod}}}

\def\Aut{\operatorname{Aut}}

\def\Der{\operatorname{Der}}
\def\Hol{\operatorname{Hol}}
\def\RD{\operatorname{D}}

\def\Im{\operatorname{Im}}
\def\G{\mathbb{G}}
\def\N{\mathbb{N}}
\def\H{\mathbb{H}}
\def\C{\mathbb{C}}
\def\D{\mathbb{D}}
\def\A{\mathcal{A}}
\def\W{\mathcal{W}}
\def\I{\mathcal{I}}

\def\leq{\leqslant}
\def\geq{\geqslant}
\def\epsilon{\varepsilon}

\begin{document}
\title{Actor of categories internal to groups}

\author{Tunçar ŞAHAN\thanks{Tunçar ŞAHAN (e-mail : tuncarsahan@gmail.com)}}
\affil{\small{Department of Mathematics, Aksaray University, Aksaray, TURKEY}}

\date{}
\maketitle

\begin{abstract}
In this study, using the Brown-Spencer theorem and in the ligth of the works of Norrie, in the category of internal categories within groups, also called group-groupoids, we interpret the notion of actor of a crossed module over groups. Further, we construct the action of a group-groupoid on a group-groupoid. Moreover, we give the explicit construction of semi-direct product of two group-groupoids and of holomorph of group-groupoids.
\end{abstract}

{\small\noindent{\bf Key Words:} Group-groupoid, actor, action, holomorph
\\ {\bf Classification:} 18D35, 20L05, 08A35}


\section{Introduction}
Crossed modules were introduced by Whitehead in \cite{Whitehead1946,Whitehead1948,Whitehead1949}. Crosed modules model all (connected) homotopy 2-types, i.e. connected spaces with $\pi_n(X)=0$ for $n\geq 2$ whilst groups model all (connected) homotopy 1-types. Crossed module concept generalizes the concept of normal subgroup and that of module since these are standart examples of crossed modules. A crossed module consists of two group homomorphisms $\alpha\colon A\rightarrow B$ and $B\rightarrow \Aut(A)$ (i.e. an action of $B$ on $A$ which denoted by $b\cdot a$) satisfying \textbf{(i)} $\alpha(b\cdot a)=b+\alpha(a)-b$ and \textbf{(ii)} $ \alpha(a)\cdot a_1 = a+a_1-a$ for all $a,a_1 \in A$ and $b\in B$. See \cite{Brown2011,Brown1982a,Huebschmann1980,Loday1978,Lue1966} for applications of crossed modules in other branches of mathematics.

Brown and Spencer \cite{Brown1976} proved that the category of group objects in the category of small categories, which is also called the category of group-groupoids, and the category of crossed modules are equivalent. This equivalence makes it possible to interpret a notion or a concept or a problem given in one of the above-mentioned categories in another. See for example \cite{Mucuk2018,Mucuk2015a}. It also has been shown that group-groupoids are in fact internal categories within the category of groups. The study of internal category theory was continued in the work of Datuashvili \cite{Datuashvili2002}. Cohomology theory of internal categories within certain algebraic categories which are called category of groups with operations is developed in \cite{Datuashvili1995} (see also \cite{Porter1987} for more information on internal categories in categories of groups with operations). The  equivalence of the categories  in \cite{Brown1976} enables us to generalize some results on group-groupoids to more general internal groupoids for a certain  algebraic category $\C$ (see for example \cite{Akiz2013}, \cite{Mucuk2015},  \cite{Mucuk2014}  and \cite{Mucuk2011}).

Let $G$ be a group and $\Aut(G)$ the group of automorphisms of $G$. It is a well-known fact that there exist a canonical group homomorphism $\varphi\colon G\rightarrow \Aut(G)$ which assigns an element of $G$ to its corresponding inner automorphism whose kernel is the centre of $G$.

In the category of groups, if there exist a short exact sequence of groups
\[\xymatrix@R=13mm@C=13mm{
	0 \ar[r] & N \ar@{^{(}->}[r]	& G \ar@{->>}[r] & H \ar[r]  & 0
}\]
such that $G$ acts on $N$, i.e. there exist a group homomorphism $G\rightarrow\Aut(N)$, then this action give rise to the following commutative diagram

\[\xymatrix@R=13mm@C=13mm{
	0 \ar[r] & N \ar@{^{(}->}[r] \ar[d]	& G \ar@{->>}[r] \ar[d] & H \ar[r] 
	\ar[d] & 0 \\
	0 \ar[r] & \operatorname{Inn}(N) \ar@{^{(}->}[r]	& \Aut(N) \ar@{->>}[r] & \operatorname{Out}(N) \ar[r] 
	& 0
}\]

The main object of this presented paper is to construct, in the category of group-groupoids, an object which is called the actor, that plays analogous role to automorphism group in the category of groups. Similar construction is studied for crossed modules over groups by Lue in \cite{Lue1979} and developed by Norrie in \cite{Norrie1987}. Norrie called this structure by actor crossed module.

The problem is finding an analogous object to $\Aut(N)$ in the category of group-groupoids that completes similar diagram to one given above. That object will be called actor group-groupoid.

Using actor group-groupoids it is possible to define the notions of centre, semi-direct product and holomorph in the category of group-groupoids. Furthermore, we define abelian and complete group-groupoids using the notion of centre for group-groupoids.

\section{Preliminaries}

A category $\C$ consist of a class $C_1$ of morphisms (or arrows), a class $C_0$ of objects, initial and final point maps $d_0,d_1\colon C_1\rightarrow C_0$, object inclusion (identity morphism) map  $\varepsilon\colon C_0\rightarrow C_1$ and a partial composition map $m\colon C_1 {_{d_0}\times_{d_1}} C_1 \rightarrow C_1$ denoted by $m(b,a)=b\circ a$ for all $(b,a)\in C_1 {_{d_0}\times_{d_1}} C_1$ where $C_1 {_{d_0}\times_{d_1}} C_1$ is the pullback of $d_0$ and $d_1$. These are subject to the following.
\begin{enumerate}[label={\textbf{(\roman{*})}}]
	\item\label{c1} $d_0\varepsilon =d_1\varepsilon ={{1}_{{{C}_{0}}}}$;
	\item\label{c2} $d_0m=d_0{{\pi }_{2}}$, $d_1m=d_1{{\pi }_{1}}$ ;
	\item\label{c3} $m\left( {{1}_{{{C}_{1}}}}\times m \right)=m\left( m\times {{1}_{{{C}_{1}}}} \right)$ and
	\item\label{c4} $m\left( \varepsilon d_0,{{1}_{{{C}_{1}}}} \right)=m\left( {{1}_{{{C}_{1}}}},\varepsilon d_1 \right)={{1}_{{{C}_{1}}}}$.
\end{enumerate}	

\[\xymatrix@R=10mm@C=10mm{
	C_1 {_{d_0}\times_{d_1}} C_1 \ar[r]^-{m}  & C_1 \ar@<.3ex>[r]^-{d_0} \ar@<-.3ex>[r]_-{d_1} 
	 & C_0  \ar@/_/@<-1ex>[l]_-{\varepsilon} 
}\]

In a category $\C$, a morphism $a\in C_1$ will be denoted by $a\colon x\rightarrow y$ where $x=d_0(a)$ and $y=d_1(a)$. For any pair of objects $x,y\in C_0$, $C_1(x,y)$ will stand for the set $d_{0}^{-1}(x)\cap d_{1}^{-1}(y)$. We denote $\varepsilon(x)$ with $1_x$ for $x\in C_0$. 

A groupoid is a category whose morphisms are invertible up to partial composition, i.e. are isomorphisms. That is, a category $\C$ is called a groupoid if for all $a\in C_1(x,y)$ there exist a unique morphism $a^{-1}\in C_1(y,x)$ such that $a^{-1}\circ a=1_x$ and $a\circ a^{-1}=1_y$.

A morphisms $f=(f_1,f_0)$ of a category $\C$ to $\D$, which is called a functor, is a pair of maps $f_1\colon C_1\rightarrow D_1$ and $f_0\colon C_0\rightarrow D_0$ such that the following diagram commutes.

\[\xymatrix@R=10mm@C=10mm{
	C_1 {_{d_0}\times_{d_1}} C_1 \ar[r]^-{m} \ar[d]_{f_1\times f_1}  & C_1 \ar[d]_{f_1} \ar@<.3ex>[r]^-{d_0} \ar@<-.3ex>[r]_-{d_1} 
	& C_0 \ar[d]^{f_0}  \ar@/_/@<-1ex>[l]_-{\varepsilon} 
	\\	
	D_1 {_{d_0}\times_{d_1}} D_1 \ar[r]^-{m}  & D_1 \ar@<.3ex>[r]^-{d_0} \ar@<-.3ex>[r]_-{d_1} 
	& C_0  \ar@/^/@<1ex>[l]^-{\varepsilon} 
}\]

It is straightforward that for a functor $f=(f_1,f_0)\colon\C\rightarrow\D$, $f_0$ can be defined via $f_1$ as follows: Let $x\in C_0$. Then $f_0(x)=d_0(1_x) ~(f_0(x)=d_1(1_x))$. Hence it is sufficient to use only $f_1$ for a functor $f=(f_1,f_0)$. So we briefly use $f$ without indices to denote a functor when no confusion arise.

\begin{definition}
Let $f,g\colon \C\rightarrow \D$ be two functors. A natural transformation $\eta\colon f\Rightarrow g$ from $f$ to $g$ is a class of morphisms $\eta(x)\colon f(x)\rightarrow g(x)$ in $D_1$ for all $x\in C_0$ such that $\eta(x)\circ f(a)=g(a)\circ \eta(y)$ for any $a\in C_1(x,y)$, i.e. the following diagram commutes.
\[\xymatrix@R=10mm@C=10mm{
f(x) \ar[r]^{\eta(x)} \ar[d]_{f(a)} &  g(x) \ar[d]^{g(a)}	\\	
f(y) \ar[r]_{\eta(y)} &  g(y)	
}\]	
\end{definition}

The set of all natural transformations between functors from a category $\C$ to a category $\D$ is denoted by $\operatorname{Nat}(\C,\D)$. In the case of $\C=\D$ we briefly denote $\operatorname{Nat}(\C,\C)$ by $\operatorname{Nat}(\C)$.

If $\eta(x)\colon f(x)\rightarrow g(x)$ is an isomorphism for all $x\in C_0$ then $\eta\colon f\Rightarrow g$ is called a natural isomorphism. In this case $f$ and $g$ are said to be \textbf{naturally isomorphic}. The set of all natural isomorphisms between functors from a category $\C$ to a category $\D$ is denoted by $\mathcal{N}(\C,\D)$. In the case of $\C=\D$ we briefly denote $\mathcal{N}(\C,\C)$ by $\mathcal{N}(\C)$. 

Since each morphism in a groupoid can be invertible then we can give the following fact.

\begin{lemma}
Let $\C$ be a category and $\G$ a groupoid. Then $\Nat(\C,\G)=\mathcal{N}(\C,\G)$, i.e. every natural transformation between functors from $\C$ to $\G$ is a natural isomorphism.
\end{lemma}

Natural transformations can be composed in two ways: Vertical composition is a partial operation on the set $\Nat(\C,\D)$ of all natural transformations. Let $\eta,\zeta\in \Nat(\C,\D)$ with $\eta\colon f\Rightarrow g$, $\zeta\colon g\Rightarrow h$. Then the vertical composition $\zeta\circ_v\eta\colon f\Rightarrow h$ can be defined and for all $x\in C_0$ it is given by \[(\zeta\circ_v\eta)(x)=\zeta(x)\circ\eta(x).\] On the other hand horizontal composition is defined as follows: Let $\eta'\in\Nat(\D,\mathbb{E})$ with $\eta'\colon f'\Rightarrow g'$. Then for all $x\in C_0$, $\eta'\circ_h\eta\colon f'f \Rightarrow g'g$ is given by \[(\eta'\circ_h\eta)(x)=\eta'(g(x))\circ f'(\eta(x)).\] 

Let $\eta,\zeta\in \Nat(\C,\D)$ and $\eta',\zeta'\in \Nat(\D,\mathbb{E})$ as in the following diagram
\[\xymatrix@C+2pc{
	\C\ruppertwocell^f{\eta}
	\rlowertwocell_h{\zeta}
	\ar[r]_(.35)g & \D  \ruppertwocell^{f'}{\eta'}
	\rlowertwocell_{h'}{\zeta'}
	\ar[r]_(.35){g'} & \mathbb{E}}\]
Then horizontal and vertical compositions of natural transformations has the following interchange rule. 
\[(\zeta'\circ_h\zeta)\circ_v(\eta'\circ_h\eta)=(\zeta'\circ_v\eta')\circ_h(\zeta\circ_v\eta)\]

As similar to the homotopy of continuous functions, the homotopy of functors is defined in \cite[p.228]{Brown2006} and it has been proven in \cite{Mucuk2013} that two parallel functors are homotopic if and only if they are naturally isomorphic.

A group-groupoid $\G$ is a group object in the category of small categories. Group-groupoids are also known as internal categories within the category of groups. So group-groupoids, also called 2-groups, can be seen as 2-dimensional groups.

Let $\G$ be a group-groupoid. Then $\G$ is a groupoid such that the class of morphisms $G_1$ and the class of objects $G_0$ has group structures and the category structural maps $d_0,d_1,\varepsilon$ and $m$ are group homomorphisms. A group-groupoid $\G$ will be denoted by $\G=(G_1,G_0)$ for short when no confusion arise.

\begin{example}\begin{enumerate}[label={\textbf{(\roman{*})}}]
	\item Let $X$ be a topological group. Then the set $\pi X$ of all homotopy classes of paths in $X$ has a group-groupoid structure where the group of objects is $X$ \cite{Brown1976}.
	\item Let $G$ be a group. Then $(G\times G,G)$ is a group-groupoid where $G_1(x,y)=(G\times G)(x,y)=\{(x,y)\}$ for all $x,y\in G$.
	\item Let $G$ be a group. Then $(G,G)$ is a group-groupoid where
	\[G_1(x,y)=G(x,y) =\left\{\begin{array}{ccl}
	\{x\} &  ,  &  \mbox {$x=y$}\\
	\emptyset &  ,  &  \mbox{$x\neq y$}
	\end{array}
	\right. \]
	for all $x,y\in G$. This group-groupoid is called discrete.
	\end{enumerate}
\end{example}

A morphism of group-groupoids $f=(f_1,f_0)$ is a functor such that each component of $f$ is a group homomorphism. Thus one can define the category of group-groupoids which is denoted by $\GpGd$.

Recall from \cite{Brown1976} that $m$ being a group homomorphism implies that 
\[(b\circ a)+(d\circ c)=(b+d)\circ (a+c)\]
whenever left side (hence both sides) of the equation make sense for all $a,b,c,d\in G_1$. This equation is called the interchange law. 

As a consequence of the interchange law in a group-groupoid $\G$ the partial composition can be given in terms of group operation as follow: Let $a\in G_1(x,y)$ and $b\in G_1(y,z)$ be two morphism in $G_1$. Then 
\[\begin{array}{rl}
b\circ a & =\left( b+0 \right)\circ \left( 1_y +\left( -1_y +a \right) \right) \\
& =\left( b\circ 1_y \right)+\left( 0\circ \left( -1_y +a \right) \right) \\
& =b-1_y+a
\end{array}\]
and similarly $b\circ a=a-1_y+b$ \cite{Brown1976}. Here, if $y=0$ then $a+b=b+a$, that is, the elements of $\Ker d_0$ and of $\Ker d_1$ are commute. Also for $a\in G_1(x,y)$, $a^{-1}=1_x-a+1_y$ is the inverse of $a$ up to the partial composition \cite{Brown1976}.

A crossed module over groups is a group homomorphism $\alpha\colon A\rightarrow B$ with an (left) action of $B$ on $A$ denoted by $b\cdot a$ such that 
\begin{enumerate}[label={\textbf{(CM\arabic{*})}}, leftmargin=1.5cm]
	\item $\alpha(b\cdot a)=b+\alpha(a)-b$ and
	\item $\alpha(a)\cdot a_1=a+a_1-a$
\end{enumerate}	
for all $a,a_1\in A$ and $b\in B$. 

\begin{example}
	Following homomorphisms are standart examples of crossed modules.
	\begin{enumerate}[label={\textbf{(\roman{*})}}]
		\item Let $X$ be a topological space, $A\subset X$ and $x\in A$. Then the boundary map \[\rho\colon \pi_2(X,A,x)\rightarrow \pi_1(X,x)\] from the second relative homotopy group ${{\pi }_{2}}\left( X,A,x \right)$ to the fundamental group ${{\pi }_{1}}\left( X,x \right)$ at $x\in X$ is a crossed module.
		\item Let $G$ be a group and $N$ a normal subgroup of $G$. Then the inclusion function $N\xrightarrow{inc}G$ is a crossed module where the action of $G$ on $N$ is conjugation.
		\item Let $G$ be a group. Then the inner automorphism map $G\rightarrow \operatorname{Aut}(G)$ is a crossed module. Here the action is given by $\psi\cdot g=\psi(g)$ for $\psi\in\operatorname{Aut}(G)$ and $g\in G$.
		\item Given any $G$-module, $M$, the trivial homomorphism $0:M\to G$ is a crossed $G$-module with the given action of $G$ on $M$.
	\end{enumerate}
\end{example}

A morphism $f=\left\langle {{f}_{A}},{{f}_{B}} \right\rangle $ of crossed modules from $\left( A,B,\alpha  \right)$ to $\left( A',B',\alpha ' \right)$ is a pair of group homomorphisms ${{f}_{A}}:A\to A'$ and ${{f}_{B}}:B\to B'$ such that ${{f}_{B}}\alpha =\alpha '{{f}_{A}}$ and ${{f}_{A}}\left( b\cdot a \right)={{f}_{B}}\left( b \right)\cdot {{f}_{A}}\left( a \right)$ for all $a\in A$ and $b\in B$.
\[\xymatrix{ A \ar[r]^{\alpha} \ar[d]_{f_A}  &  B \ar[d]^{f_B} \\  A' \ar[r]_{\alpha'} & B' }\]

Crossed modules form a category with morphisms defined above. The category of crossed modules is denoted by $\XMod$.

Following theorem was given in \cite{Brown1976}. We sketch the proof since we need some details later.

\begin{theorem}[\textbf{Brown \& Spencer Theorem}]\cite{Brown1976}\label{BST}
	The category $\GpGd$ of group-groupoids and the category $\XMod$ of crossed modules are equivalent.
\end{theorem}
\begin{proof}
	Define a functor
	\[\varphi :\GpGd\to \XMod\]
	as follows: Let $\G$ be a group-groupoid. Then $\varphi \left( \G \right)=\left( A,B,\alpha  \right)$ is a crossed modules where $A=\ker s$, $B={{G}_{0}}$, $\alpha $ is the restriction of $t$ and the action of $B$ on $A$ is given by $x\cdot a={{1}_{x}}+a-{{1}_{x}}$.
	
	Conversely, define a functor
	\[\psi :\XMod\to \GpGd\]
	as follows: Let $\left( A,B,\alpha  \right)$ be a crossed module. Then the semi-direct product group $A\rtimes B$ is a group-groupoid on $B$ where $s\left( a,b \right)=b$, $t\left( a,b \right)=\alpha \left( a \right)+b$, $\varepsilon \left( b \right)=\left( 0,b \right)$ and the composition is
	$\left( a',b' \right)\circ \left( a,b \right)=\left( a'+a,b \right)$
	where $b'=\alpha \left( a \right)+b$.
	
	Other details are straightforward so is omitted.
\end{proof}

\section{Actor of a group-groupoid}

Norrie \cite{Norrie1990} defined actor crossed module for a given crossed module using regular derivations of that crossed module and showed how actor crossed modules provides an analogue of automorphisms groups of groups. In this section, in the light of the Brown \& Spencer Theorem \cite[Theorem 1]{Brown1976} we define actor group-groupoid, group-groupoid actions and semi-direct product of group-groupoids. First we recall the construction of actor crossed module from \cite{Norrie1990}.

The notion of derivations first appears in the work of Whitehead \cite{Whitehead1948} under the name of crossed homomorphisms (see also \cite{Norrie1990}). Let $(A,B,\alpha)$ be a crossed module. A derivation of $(A,B,\alpha)$ is a map $d\colon B\rightarrow A$ such that \[d(b+b_1)=d(b)+b\cdot d(b_1)\] for all $b,b_1\in B$. Any derivation of $(A,B,\alpha)$ defines endomorphisms $\theta_d$ and $\sigma_d$ on $A$ and $B$ respectively, as follows:
\[\theta_d(a)=d\alpha(a)+a  \quad \text{and} \quad \sigma_d(b)=\alpha d(b)+b\]


It is easy to see that $(\theta_d,\sigma_d)\colon (A,B,\alpha)\rightarrow (A,B,\alpha)$ is a crossed module morphism and $\theta_d(d)=d(\sigma_d)$. All derivations of $(A,B,\alpha)$ are denoted by $\Der(B,A)$. Whitehead defined a multiplication on $\Der(B,A)$ as follows: Let $d_1,d_1\in\Der(B,A)$ then $d=d_1\circ d_2$ where
\[d(b)=d_1\sigma_{d_2}(b)+d_2(b) \ \ \ \ \	(=\theta_{d_1}d_2(b)+d_1(b)).\]

By this multiplication $(\Der(B,A),\circ)$ becomes a monoid where the identity derivation is zero morphism, i.e. $d\colon B\rightarrow A, b\mapsto d(b)=0$. Furthermore $\theta_d=\theta_{d_1}\theta_{d_2}$ and $\sigma_d=\sigma_{d_1}\sigma_{d_2}$. Group of units of $\Der(B,A)$ is called \emph{Whitehead group} and denoted by $\RD(B,A)$. These units are called \emph{regular derivations}. 

Following proposition which is given in \cite{Norrie1990} is a combined result from \cite{Whitehead1948} and \cite{Lue1979}.

\begin{proposition}\cite{Norrie1990}
	Let $(A,B,\alpha)$ be a crossed module. Then the followings are equivalent.
	\begin{enumerate}[label=\textbf{(\roman{*})}, leftmargin=1cm]
		\item $d\in\RD(B,A)$;
		\item $\theta\in\Aut A$;
		\item $\sigma\in\Aut B$.
	\end{enumerate}
\end{proposition}

Actor of a crossed module $(A,B,\alpha)$ is defined in \cite{Norrie1990} to be the crossed module \[(\RD(B,A),\Aut(A,B,\alpha),\Delta)\] where $\Delta(d)=\left\langle \theta_d,\sigma_d\right\rangle $ and the action of $\Aut(A,B,\alpha)$ on $\RD(B,A)$ is given by $\left\langle f,g \right\rangle \cdot d = fdg^{-1}$ for all $\left\langle f,g \right\rangle\in\Aut(A,B,\alpha)$ and $d\in \RD(B,A)$. Actor crossed module of $(A,B,\alpha)$ is denoted by $\A(A,B,\alpha)$.

Now we obtain analogous construction for group-groupoids. First we recall some preliminary definitions and properties for group-groupoids.

\begin{definition}\cite{Mucuk2015a}
Let $\G$ be a group-groupoid and $\H$ a subgroupoid of $\G$ such that $H_1$ is a subgroup of $G_1$. Then $\H$ is called a \textbf{subgroup-groupoid} of $\G$ and this denoted by $\H\leq \G$.
\end{definition}

\begin{definition}\label{nsgpgd}\cite{Mucuk2015a}
Let $\G$ be a group-groupoid and $\N$ a subgroupoid of $\G$ such that $N_1$ is a normal subgroup of $G_1$. Then $\N$ is called a \textbf{normal subgroup-groupoid} of $\G$ and this denoted by $\H\lhd \G$.
\end{definition}

\begin{definition}\cite{Mucuk2015a}
Let $\G$ be a group-groupoid and $\H$ a normal subgroup-groupoid of $\G$. Then the quotient group $G_1/N_1$ is a group-groupoid on the quotient group $G_0/N_0$. This group-groupoid is called the \textbf{quotient group-groupoid} and denoted by $\G/\N$.
\end{definition}

A morphism $f=(f_1,f_0)\colon \G\rightarrow \G'$ of group-groupoids is called an isomorphism (monomorphism, epimorphism and automorphism) if $f_1$ and $f_0$ are both isomorphisms (monomorphisms, epimorphisms and automorphisms). The group of all automorphisms of $\G$ is denoted by $\Aut(\G)$. The kernel of a group-groupoid morphism $f=(f_1,f_0)\colon \G\rightarrow \G'$ is the normal subgroup-groupoid $\Ker f=(\Ker f_1,\Ker f_0)$ of $\G$. The image of $f=(f_1,f_0)$ is the subgroup-groupoid $\Im f=(\Im f_1,\Im f_0)$ of $\G'$.

Let $\G$ be a group-groupoid and $\H,\mathbb{K}\leq\G$. Then $[\H,\mathbb{K}]=\left( [H_1,K_1],[H_0,K_0]\right) $ is a subgroup-groupoid of $\G$ and called the \textbf{commutator} of $\H$ and $\mathbb{K}$. In particular $[\G,\G]$, which is denoted by $\G'$, is called the \textbf{commutator subgroup-groupoid} or \textbf{derived subgroup-groupoid} of $\G$. It is easy to see from Definition \ref{nsgpgd} that $\G'$ is a normal subgroup-groupoid of $\G$ since $[G_1,G_1]$ is a normal subgroup of $G_1$.

A natural transformation $\eta\in\Nat(\G,\H)$ between group-groupoid morphisms from $\G$ to $\H$ is a usual natural transformation such that \[\eta(x_1+x_2)=\eta(x_1)+\eta(x_2)\] for all $x_1,x_2\in G_0$.

For a group-groupoid $\G$, $\Nat(\G)$ is a monoid where the monoid operation is horizontal composition of natural transformations. Here the identity element of $\Nat(G)$ is $1_{1_{\G}}$. A natural transformation $\eta$ in $\Nat(G)$ is called \textbf{regular} if it has an inverse up to horizontal composition. All regular natural transformations forms a group. This group is denoted by $\W(\G)$. 

\begin{proposition}
Let $\eta\in\Nat(\G)$ with $\eta\colon f\Rightarrow g$. Then $\eta\in\W(\G)$ if and only if $f,g\in\Aut(\G)$.
\end{proposition}

\begin{proof}
If $\eta\in\W(\G)$ then there exist $\eta'\in\Nat(\G)$ with $\eta'\colon f'\Rightarrow g'$ such that $\eta\circ_h\eta'=1_{1_{\G}}=\eta'\circ_h\eta$. So $ff'=1_{\G}=f'f$ and $gg'=1_{\G}=g'g$. Thus $f'=f^{-1}$ and $g'=g^{-1}$, i.e. $f,g\in\Aut(\G)$.

Conversely, let $f,g\in\Aut(\G)$. Then horizontal inverse of $\eta$ is given by
\[\eta^{-h}(x)=\left[ f^{-1} \left( \eta\left( g^{-1}\left( x \right) \right)  \right)  \right]^{-1}. \] Thus $\eta\in\W(\G)$. This completes the proof.
\end{proof}

In the following lemma we define an automorphism using regular natural transformations which will be useful later.

\begin{lemma}\label{auteta}
Let $\eta\in\W(\G)$ with $\eta\colon f\Rightarrow g$. Then the map defined by
\[\begin{array}{rcccl}
F^{\eta} & : & G_1 & \rightarrow  & G_1  \\
{} & {} & a & \mapsto  & F^{\eta}(a)=g(a)\circ\eta(d_0(a))=\eta(d_1(a))\circ f(a) \\
\end{array}\]
is an automorphism of $G_1$.
\end{lemma}

\begin{proof}
First we show that $F^{\eta}$ is a group homomorphism. Let $a\in G_1(x,y)$  and $a_1\in G_1(x_1,y_1)$. Then
\[\begin{array}{rcl}
F^{\eta}(a+a_1) &=& g(a+a_1)\circ\eta(x+x_1)\\ 
&=&\left( g(a) + g(a_1)\right) \circ\left( \eta(x) + \eta(x_1)\right)  \\
&=&\left( g(a) \circ \eta(x) \right) + \left(g(a_1) \circ \eta(x_1)\right)  \\
&=& F^{\eta}(a)+F^{\eta}(a_1).
\end{array}\]

Now let $a,a_1\in G_1$ such that $a\neq a_1$. It is obvious that $F^{\eta}(a)\neq F^{\eta}(a_1)$ if at least one of the initial or final points of $a$ and $a_1$ is different since the initial or final points of $F^{\eta}(a)$ and $F^{\eta}(a_1)$ is different in that case. Thus let $a,a_1\in G_1(x,y)$. Then 
\[F^{\eta}(a) = g(a)\circ\eta(x)\neq g(a_1)\circ\eta(x) = F^{\eta}(a_1)\]
since $g$ is an isomorphism. Hence $F^{\eta}$ is a monomorphism.

Finally, let $b\in G_1(m,n)$. If we set $a=g^{-1}\left( b\circ \eta(f^{-1}(m))^{-1} \right)$ then $a\in G_1(f^{-1}(m),g^{-1}(n))$ and
\[\begin{array}{rcl}
F^{\eta}(a) &=& g(a)\circ\eta(f^{-1}(m))\\ 
&=& \left( b\circ \eta(f^{-1}(m))^{-1}\right)  \circ \eta(f^{-1}(m))\\
&=& b.
\end{array}\]
Hence $F^{\eta}$ is an epimorphism. So this completes the proof.
\end{proof}

\[\xymatrix@R=13mm@C=13mm{
	f(x) \ar[r]^{\eta(x)} \ar[d]_{f(a)} \ar@{..>}[dr]|{F^{\eta}(a)} &  g(x) \ar[d]^{g(a)}	\\	
	f(y) \ar[r]_{\eta(y)} &  g(y)	
}\]

\begin{remark}\label{remF}
It is easy to see that $F^{\eta}(1_x)=\eta(x)$ for all $x\in G_0$. Also $F^{\left( 1_{1_{\G}}\right) }=1_{G_{1}}$.
\end{remark}

\begin{lemma}\label{mono}
Let $\eta,\tau\in \W(\G)$. If $\eta\neq \tau$ then $F^{\eta}\neq F^{\tau}$.
\end{lemma}

\begin{proof}
If $\eta\neq \tau$ then there exist at least one object $x\in G_0$ such that $\eta(x) \neq \tau(x)$. Then by Remark \ref{remF}  \[F^{\eta}(1_x)=\eta(x)\neq \tau(x)=F^{\tau}(1_x).\] This completes the proof.
\end{proof}

As a result of this lemma, first of all if $F^{\eta}= F^{\tau}$ for some $\eta,\tau\in\W(\G)$ then $\eta=\tau$. Also we can say that the map which assigns a regular transformation $\eta$ to an automorphism $F^{\eta}$ of $G_1$ is injective.

We will now examine how this automorphism behaves under horizontal and vertical compositions of natural transformations. Let $\eta,\eta',\tau\in\W(\G)$ with $\eta\colon f\Rightarrow g$, $\eta'\colon f'\Rightarrow g'$ and $\tau\colon g\Rightarrow h$. Then for any $a\in G_1(x,y)$

\[F^{\left( \tau\circ_v \eta\right) }(a)=F^{\tau}(a)\circ \eta(x) = \left( h(a)\circ \tau(x)\right)  \circ \eta(x)\]
or equivalently
\[F^{\left( \tau\circ_v \eta\right) }(a)=\tau(y) \circ F^{\eta}(a) = \tau(y) \circ \left( \eta(y)\circ f(a)\right).\]

On the other side 
\[F^{\left( \eta'\circ_h \eta\right) }(a)=\eta'(g(y))\circ f'(F^{\eta}(a))=F^{\eta'}(g(a))\circ f'(\eta(x)).\]

\begin{corollary}\label{lemhorcomF}
Let $\eta,\eta'\in\W(\G)$. Then $F^{\left( \eta'\circ_h \eta\right)}=F^{\eta'}\circ F^{\eta}$.
\end{corollary}

\begin{proposition}
Let $\G$ be a group-groupoid. Then $(\W(\G),\Aut(\G),\w{d_0},\w{d_1},\w{\varepsilon},\w{m})$ is equipped with a group-groupoid structure where $\w{d_0}(\eta)=f$, $\w{d_1}(\eta)=g$, $\w{\varepsilon}(f)=1_f$ and $\eta'~\w{\circ}~\eta=\eta'\circ_v\eta$ for all $\eta\colon f\Rightarrow g,\eta'\colon g\Rightarrow h \in\W(\G)$.	
\end{proposition}
\[\xymatrix@R=10mm@C=10mm{
	\W(\G) {_{\w{d_0}}\times_{\w{d_1}}} \W(\G) \ar[r]^-{\w{m}}  & \W(\G) \ar@<.3ex>[r]^-{\w{d_0}} \ar@<-.3ex>[r]_-{\w{d_1}} 
	& \Aut(\G)  \ar@(u,u)[l]_-{\w{\varepsilon}} 
}\]

Here the interchange law comes from the well-known identiy \[(\eta'\circ_v\eta)\circ_h(\eta_1'\circ_v\eta_1)=(\eta'\circ_h\eta_1')\circ_v(\eta\circ_h\eta_1)\] for horizontal and vertical compositiom of natural transformations whenever one side (hence both sides) of the equation make sense. 

\begin{definition}
Let $\G$ be a group-groupoid. Then the group-groupoid $(\W(\G),\Aut(\G))$ defined above is called the \textbf{actor group-groupoid} of $\G$ and denoted by $\A(\G)$.
\end{definition}

\begin{example}\begin{enumerate}[label={\textbf{(\roman{*})}}]
\item Let $X$ be a topological group. Then $\A(\pi X)\cong(\mathcal{H}(X),\Aut(X))$ where $\Aut(X)$ is the set of all topological group automorphisms of $X$ and $\mathcal{H}(X)$ is the set of all homotopies between the topological group automorphisms of $X$.
\item We know that $\G=(G\times G,G)$ is a group-groupoid for a group $G$. In this case $\W(\G)\cong \Aut(G)\times\Aut(G)$ since for each pair $(f,g)$ of automorphisms of $\G$ there exist exactly one regular natural transformation form $f$ to $g$. Thus $\A(G)\cong(\Aut(G)\times\Aut(G),\Aut(G))$.
\item Let $G$ be a group. Then for the group-groupoid $\G=(G,G)$, $\W(\G)\cong\Aut(G)$ since $\W(\G)$ contains only $1_f$ for all $f\in\Aut(\G)$. Thus the actor group-groupoid $\A(\G)\cong(\Aut(G),\Aut(G))$. 
\end{enumerate}
\end{example}

Let $\G$ be a group-groupoid, $\A(\G)$ the actor group-groupoid of $\G$ and $(A,B,\alpha)$ the corresponding crossed module to $\G$ according to Theorem \ref{BST}. So $A=\Ker d_0$, $B=G_0$ and $\alpha=d_1|_{A}$. Here we can define two group homomorphisms
\[\begin{array}{rcccl}
\xi & : & \Ker \w{d_0}  & \rightarrow  & \RD(B,A)  \\
{} & {} & \eta & \mapsto  & \xi(\eta)=d_{\eta} \\
\end{array}\]
where $d_{\eta}(x)$ is given by $d_{\eta}(x)=\eta(x)-1_x$ for $x\in B$ and
\[\begin{array}{rcccl}
\lambda & : & \Aut(\G)  & \rightarrow  & \Aut(A,B,\alpha)  \\
{} & {} & f & \mapsto  & \lambda(f)=\left\langle f_1|_{A} , f_0 \right\rangle.  \\
\end{array}\]
It is easy to see that these morphisms are isomorphisms. Here the inverse of $\xi$ is given by $\xi^{-1}(d)=\eta_d$ where $\eta_d(x)=\eta(x)+1_x$ for all $x\in G_0$ and the inverse of $\lambda$ is given by $\lambda^{-1}(\left\langle f_A,f_B \right\rangle )=(f_A\times f_B,f_B)$ for all $\left\langle f_A,f_B \right\rangle\in \Aut(A,B,\alpha)$. 

\begin{theorem}\label{isoact}
$\left\langle \xi,\lambda \right\rangle\colon \left( \Ker \w{d_0},\Aut(\G),\w{d_1}|_{\Ker \w{d_0}}\right)  \rightarrow \left( \RD(B,A),\Aut(A,B,\alpha),\Delta \right) $ is an isomorphism of crossed modules.
\end{theorem}

\begin{proof}
Let $\eta\in \Ker \w{d_0}$. Then there exist an automorphism $g$ in $\Aut(\G)$ such that $\eta\colon 1_{\G}\Rightarrow g$. So $\w{d_1}(\eta)=g$ and hence $\lambda(\w{d_1}(\eta))=\lambda(g)=\left\langle g_1|_{A} , g_0 \right\rangle$. 

On the other hand $\xi(\eta)=d_{\eta}$ and $\Delta(\xi(\eta))=\Delta(d_{\eta})=\left\langle \theta_{d_{\eta}}, \sigma_{d_{\eta}} \right\rangle $. Now let compute the morphisms $\theta_{d_{\eta}}$ and $\sigma_{d_{\eta}}$.

\[\begin{array}{rcl}
\theta_{d_{\eta}}(a) &=& d_{\eta}(d_1(a))+a\\ 
&=& \eta(d_1(a))-1_{d_1(a)}+a \\
&=& \eta(d_1(a))\circ a \\
&=& g_1(a)\circ \eta(d_0(a)) \\
&=& g_1(a)\circ \eta(0) \\
&=& g_1(a)
\end{array}\]
and
\[\begin{array}{rcl}
\sigma_{d_{\eta}}(x) &=& d_1(d_{\eta}(x))+x\\ 
&=& d_1(\eta(x)-1_x)+x \\
&=& d_1(\eta(x))-d_1(1_x)+x \\
&=& g_0(x)-x+x \\
&=& g_0(x).
\end{array}\]
That is $\Delta(d_{\eta})=\left\langle \theta_{d_{\eta}}, \sigma_{d_{\eta}} \right\rangle =\left\langle g_1|_{A} , g_0 \right\rangle $. So we obtain the equality $\lambda\w{d_1}|_{\Ker \w{d_0}}=\Delta\xi$.

Now we need to show that $\xi(f\cdot\eta)=\lambda(f)\cdot\xi(\eta)$. By Theorem \ref{BST} $f\cdot\eta=1_f\circ_h \eta \circ_h 1_f^{-1}=1_f\circ_h \eta \circ_h 1_{f^{-1}}$. Then $\xi(f\cdot\eta)=d_{f\cdot\eta}=d_{1_f\circ_h \eta \circ_h 1_{f^{-1}}}$ and for all $x\in B$
\[\begin{array}{rcl}
d_{\left( 1_f\circ_h \eta \circ_h 1_{f^{-1}}\right) } &=& \left( 1_f\circ_h \eta \circ_h 1_{f^{-1}}\right) (x)-1_x\\ 
&=& f_1(\eta(f_{0}^{-1}(x)))-1_x.
\end{array}\]
On the other hand $\lambda(f)\cdot\xi(\eta)=\left\langle f_1 ,f_0 \right\rangle \cdot d_{\eta} = f_1 d_{\eta} f_0^{-1}$ and for all $x\in B$
\[\begin{array}{rcl}
\left( f_1 d_{\eta} f_0^{-1} \right) (x) &=& f_1 \left( d_{\eta} \left( f_0^{-1} (x)\right) \right) \\ 
&=& f_1 \left( \eta\left( f_0^{-1} (x)\right) - 1_{f_0^{-1} (x)}\right) \\
&=& f_1 \left( \eta\left( f_0^{-1} (x)\right)\right) - f_1 \left(1_{f_0^{-1} (x)}\right) \\
&=& f_1(\eta(f_{0}^{-1}(x)))-1_x.
\end{array}\]
So $\xi(f\cdot\eta)=\lambda(f)\cdot\xi(\eta)$ and hence $\left\langle \xi,\lambda \right\rangle$ is a crossed module morphism. We know that $\xi$ and $\lambda$ are group isomorphisms. So this completes the proof.
\end{proof}

\begin{corollary}
Let $\G$ be a group-groupoid and $(A,B,\alpha)$ the corresponding crossed module to $\G$. Then $\A(A,B,\alpha)=(\RD(B,A),\Aut(A,B,\alpha),\Delta)$ is isomorphic to the corresponding crossed module to $\A(\G)$.
\end{corollary}

Theorem \ref{isoact} states that the definition of actor group-groupoid is compatible with the one given in \cite{Norrie1990} by Norrie for crossed modules.

\subsection{Center of a group-groupoid}

Now let us consider the group-groupoid automorphisms $f^{x}:=(f_{1}^{x},f_{0}^{x})$ of $\G$ given by
\[\begin{array}{rcccl}
f_{1}^{x} & : & G_1 & \rightarrow  & G_1  \\
{} & {} & b & \mapsto  & f_{1}^{x}(b)=1_x+b-1_x \\
\end{array}\]
on the group of objects and by
\[\begin{array}{rcccl}
f_{0}^{x} & : & G_0 & \rightarrow  & G_0  \\
{} & {} & z & \mapsto  & f_{0}^{x}(z)=x+z-x \\
\end{array}\]
on the group of morphisms for all $x\in G_0$. Using these automorphisms we can define a canonical group-groupoid morphism $\varphi=(\varphi_1,\varphi_0)\colon \G\rightarrow \A(\G)$ as follows: Let $a\colon x\rightarrow y$ be a morphism in $G_1$ and $z$ an object in $G_0$. Then $\varphi_1(a)=\eta^a\colon f^{x}\Rightarrow f^{y}$ where $\eta^a(z)=a+1_z-a$ and $\varphi_0(z)=f^z$. It is easy to see that $\eta^a\in \W(\G)$ for all $a\in G_1$ and $\varphi$ is a group-groupoid morphism. Here the automorphism $F^{\eta^{a}}$ introduced in Lemma \ref{auteta} corresponding to $\eta^{a}$ is given by $F^{\eta^{a}}(b)=a+b-a$, i.e. the conjugation. 

Note that for all $x,y\in G_0$, $f^x f^y=f^{x+y}$ and for all $a,b\in G_1$, $\eta^{a}\circ_h\eta^{b}=\eta^{a+b}$. Hence the inverse $(f^x)^{-1}$ of $f^x$ is equal to $f^{-x}$ and the horizontal inverse $\left( \eta^{a}\right)^{-h}$ of $\eta^{a}$ is equal to $\eta^{-a}$.

We define the centre of a group-groupoid $\G$ to be the kernel $Z(\G)$ of $\varphi$. Thus \[\left( Z\left(\G \right)  \right)_1 := \Ker \varphi_1=\{a\in G_1 ~|~ a+1_x=1_x+a \text{ for all } x\in G_0 \}\] and 
\[\left( Z\left(\G \right)  \right)_0 := \Ker \varphi_0=\{x\in G_0 ~|~ a+1_x=1_x+a \text{ for all } a\in G_1\}.\] 

Here note that $\left( Z\left(\G \right)  \right)_0$ is a subgroup of the centre $Z(G_0)$ of $G_0$.

Since $Z(\G)$ is the kernel of the group-groupoid morphism $\varphi$ then we can give the following Lemma.

\begin{lemma}
Let $\G$ be a group-groupoid. Then $Z(\G)$ is a normal subgroup-groupoid of $\G$.
\end{lemma}

Definition of a centre for a group-groupoid is consistent with the categorical one given by Huq \cite{Huq1968} since the corresponding crossed module to $Z(\G)$ is isomorphic to $\xi(A,B,\alpha)$ which is defined by Norrie in \cite{Norrie1990} while $(A,B,\alpha)$ is the corresponding crossed module to $\G$.

\begin{lemma}\label{lemcom}
Let $\G$ be a group-groupoid. Then the elements of $G_1$ and of $\left( Z(\G)\right)_1$ are commute.
\end{lemma}
\begin{proof}
Let $a\in \left( Z(\G)\right)_1$ and $b\in G_1$. Then $a-1_{d_0(a)}\in\Ker d_0$ and $b-1_{d_1(b)}\in\Ker d_1$ so  
\[(a-1_{d_0(a)})+(b-1_{d_1(b)})=(b-1_{d_1(b)})+(a-1_{d_0(a)}).\]
Since $d_0(a)\in \left( Z(\G)\right)_0 $ then
\[\begin{array}{rcl}
(a-1_{d_0(a)})+(b-1_{d_1(b)}) &=&(b-1_{d_1(b)})+(a-1_{d_0(a)}) \\
(a+b)+(-1_{d_0(a)}-1_{d_1(b)}) &=&(b+a)+(-1_{d_1(b)}-1_{d_0(a)}) \\
a+b &=&b+a
\end{array}\]
This completes the proof.
\end{proof}

Analogous to the group case we can introduce the notion of abelian group-groupoid using centres.

\begin{definition}\label{abgpgd}
A group-groupoid $\G$ is called \textbf{abelian} if it coincide with its centre, i.e. $\G=Z(\G)$.
\end{definition}

Following is a consequence of Lemma \ref{lemcom} and Definition \ref{abgpgd}.

\begin{corollary}\label{lemab}
A group-groupoid $\G$ is abelian if and only if $G_1$ (hence also $G_0$) is an abelian group.
\end{corollary}

By Lemma \ref{lemcom} and Corollary \ref{lemab} we can give the following corollary.

\begin{corollary}
A group-groupoid $\G$ is abelian if and only if $a+1_x=1_x+a$ for all $a\in G_1$ and $x\in G_0$.
\end{corollary}

\begin{corollary}
Let $\G$ be a group-groupoid. Then $\G/\G'$ is an abelian group-groupoid where $\G'$ is the commutator subgroup-groupoid of $\G$.
\end{corollary}

The quotient group-groupoid given above is called the \textbf{abelianization} of $\G$ and denoted by $\G^{\text{ab}}=\G/\G'$. 

\begin{example}\begin{enumerate}[label={\textbf{(\roman{*})}}]
	\item Let $X$ be an abelian topological group. Then $(\pi X,X)$ is an abelian group-groupoid.		
	\item $\G=(G\times G,G)$ is an abelian group-groupoid if and only if $G$ is an abelian group.		
	\item If $G$ is an abelian group then the group-groupoid $\G=(G,G)$ is abelian. 
	\end{enumerate}
\end{example}

Abelian group-groupoids forms a full subcategory of $\GpGd$ and this subcategory is denoted by $\AGpGd$.

The \textbf{inner actor} $\I(\G)$ of $\G$ is the image $\Im \varphi$ of the group-groupoid morphism $\varphi\colon \G\rightarrow \A(\G)$. One can see that $\I(\G)$ is a normal subgroup-groupoid of $\A(\G)$. Here, by the first isomorphism theorem for group-groupoids \cite{Mucuk2015a}, $\G/Z(\G)\cong \I(\G)$. Also the quotient group-groupoid $\A(\G)/\I(\G)$ is called the \textbf{outer actor} of $\G$ and denoted by $\mathcal{O}(\G)$. Hence we obtain the short exact sequence
\[\xymatrix@R=13mm@C=13mm{
	0 \ar[r] & Z(\G) \ar[r]^-{\ker \varphi}	& \G \ar[r]^-{\varphi} & \A(\G) \ar[r]^-{\coker \varphi} & \mathcal{O}(\G) \ar[r] & 0
}\]
of group-groupoids since $\mathcal{O}(\G)$ is the cokernel of $\varphi$.

By the definition of centre of a group-groupoid $\G$, if $Z(\G)=\bm 0$, i.e. the group-groupoid with one morphism, then $\varphi\colon \G\rightarrow\A(\G)$ is a monomorphism of group-groupoids. In this case $\G\cong\I(\G)$ and can be consider as a normal subgroup-groupoid of $\A(\G)$. 

\begin{proposition}\label{tricen}
Let $\G$ be a group-groupoid. If $Z(\G)$ is trivial group-groupoid then so is $Z(\A(\G))$.
\end{proposition}

\begin{proof}
Let $\eta\in (Z(\A(\G)))_1$. Then $\eta\circ_h\eta' = \eta'\circ_h\eta$ for all $\eta'\in \W(\G)$ by Lemma \ref{lemcom}. In particular $\eta\circ_h\eta^{a} = \eta^{a}\circ_h\eta$ for any $a\in G_1$. Then for all $b\in G_1$ 
\[\begin{array}{rcl}
F^{\left( \eta\circ_h\eta^{a}\right) }(b) &=& F^{\left( \eta^{a}\circ_h\eta\right) }(b) \\
F^{\eta}(F^{\eta^{a}}(b)) &=& F^{\eta^{a}}(F^{\eta}(b)) \\
F^{\eta}(a+b-a) &=& a+(F^{\eta}(b))-a \\
F^{\eta}(a)+F^{\eta}(b)-F^{\eta}(a) &=& a+(F^{\eta}(b))-a
\end{array}\]
and hence \[F^{\eta}(b)+(-F^{\eta}(a)+a) = (-F^{\eta}(a)+a)+(F^{\eta}(b)).\]
Because of the fact that $F^{\eta}$ is an isomorphism $(-F^{\eta}(a)+a)\in (Z(\G))_1$. Since $Z(\G)$ is trivial then $F^{\eta}(a)=a$ for all $a\in G_1$. 

On the other hand we already know that $F^{\left( 1_{1_{\G}}\right) }(a)=a$ for all $a\in G_1$. So $F^{\eta}=F^{\left( 1_{1_{\G}} \right) }$. Then by Lemma \ref{mono}, $\eta=1_{1_{\G}}$. Hence $Z(\A(\G))$ is also trivial.
\end{proof}

By Proposition \ref{tricen} for a given group-groupoid $\G$ with $Z(\G)=\bm 0$, we can construct a sequence of group-groupoids
\[\xymatrix{
	\G \ar[r] & \A(\G) \ar[r] & \A(\A(\G))=\A^2(\G) \ar[r] & \A(\A(\A(\G)))=\A^3(\G) \ar[r] & \cdots 
}\]
in which each group-groupoid embeds as a normal subgroup-groupoid in its successor. This sequence is called the \textbf{actor tower} of $\G$. 

\begin{definition}
A group-groupoid is called \textbf{complete} if $Z(\G)=\bm 0$ and the canonical morphism $\varphi\colon\G\rightarrow\A(\G)$ is an isomorphism.
\end{definition}

\begin{corollary}
If $\G$ is a complete group-groupoid then $\G\cong\I(\G)=\A(\G)$.
\end{corollary}

We know that if $Z(\G)=\bm 0$ then $\varphi\colon\G\rightarrow\A(\G)$ is a monomorphism. So a group-groupoid is complete if and only if $Z(\G)=\bm 0$ and the canonical morphism $\varphi\colon\G\rightarrow\A(\G)$ is an epimorphism.

\begin{lemma}
Let $\G$ be a 1-transitive group-groupoid, i.e. $G_1(x,y)$ is singleton for all $x,y\in G_0$. Then $\G$ is complete if and only if $G_1$ is a complete group.
\end{lemma}

Let $\G$ be a group-groupoid with $Z(\G)=\bm 0$. If $\A^n(\G)$ is complete for some $n\in\mathbb{N}$ then the actor tower stops, that is, for all $k\in\mathbb{N}$ with $k>n$, $\A^n(\G)\cong\A^k(\G)$. In this case the actor tower of $\G$ is a finite sequence
\[\xymatrix{\G \ar[r] & \A(\G) \ar[r] & \A^2(\G) \ar[r] & \cdots \ar[r] & \A^n(\G).}\]

\subsection{Actions and semi-direct products of group-groupoids}

An action of a group-groupoid $\H$ on a group-groupoid $\G$ is defined to be a group-groupoid morphism $\theta\colon\H\rightarrow\A(\G)$. Here note that $\A(\G)$ acts on $\G$ where $\theta=1_{\A(\G)}$. $\I(\G)$ also acts on $\G$ and on any normal subgroup-groupoid $\N$ of $\G$. Thus $\G$ acts on any of its normal subgroup-groupoids.

For a given action $\theta\colon\H\rightarrow\A(\G)$ one can obtain group actions of $H_1$ on $G_1$ (by Lemma \ref{auteta}) and of $H_0$ on $G_0$. These actions are compatible with the group-groupoid structural maps, that is for all $a\in G_1(x_1,x_2), a'\in G_1(x_1',x_2')$ with $x_2=x_1'$ and $b\in H_1(y_1,y_2), b'\in H_1(y_1',y_2')$ with $y_2=y_1'$
\begin{enumerate}[label={\textbf{(\roman{*})}}]
\item $d_0(b\cdot a)=d_0(b)\cdot d_0(a)=y_1\cdot x_1$,
\item $d_1(b\cdot a)=d_1(b)\cdot d_1(a)=y_2\cdot x_2$,
\item\label{iii} $1_{y_1\cdot x_1}=\varepsilon(y_1\cdot x_1)=\varepsilon(y_1)\cdot \varepsilon(x_1)=1_{y_1}\cdot 1_{x_1}$ and
\item $(b'\circ b)\cdot (a'\circ a)=(b'\cdot a')\circ (b\cdot a)$.
\end{enumerate}

Here note that, by \ref{iii}, the action of $H_0$ on $G_0$ can be obtained from that of $H_1$ on $G_1$. So it is sufficient to consider only the action of $H_1$ on $G_1$.

An extension of a group-groupoid $\N$ by a group-groupoid $\H$ as given in \cite{Temel2018} is a sequence of group-groupoid morphisms
\[\xymatrix@R=13mm@C=13mm{
	{\bm 0} \ar[r] & \N \ar@{^{(}->}[r]^-{i}	& \G \ar@{->>}[r]_-{p} & \H \ar[r] \ar@/_/[l]_{s} & {\bm 0}
}\]
where $i$ is a monomorphism, $p$ is an epimorphism, $\ker p=i(\N)$ and $ps=1_{\H}$. An action of $\G$ on $\N$ induces the following commutative diagram:

\[\xymatrix@R=13mm@C=13mm{
	{\bm 0} \ar[r] & \N \ar@{^{(}->}[r] \ar[d]	& \G \ar@{->>}[r] \ar[d] & \H \ar[r] 
	\ar[d] & {\bm 0} \\
	{\bm 0} \ar[r] & \I(\N) \ar@{^{(}->}[r]	& \A(\N) \ar@{->>}[r] & \mathcal{O}(\N) \ar[r] 
	& {\bm 0}.
}\]

Let the group-groupoid $\H$ acts on the group-groupoid $\G$. So we have a group-groupoid morphism $\theta\colon\H\rightarrow\A(\G)$. Then there exist group actions $\theta_1$ of $H_1$ on $G_1$ and $\theta_0$ of $H_0$ on $G_0$. Using these actions we can construct the semi-direct product groups $G_1\rtimes_{\theta_1}H_1$ and $G_0\rtimes_{\theta_0}H_0$. Moreover, $\left( G_1\rtimes_{\theta_1}H_1 , G_0\rtimes_{\theta_0}H_0 \right) $ has a natural group-groupoid structure. We call this group-groupoid the semi-direct product of $\G$ and $\H$ relative to $\theta$ and denote this group-groupoid by $\G\rtimes_{\theta}\H$ or briefly by $\G\rtimes\H$ when no confusion arise.

Analogous to group case we can define internal semi-direct product of group-groupoids. Let $\G$ be a group-groupoid with a normal subgroup-groupoid $\N$ and a subgroup-groupoid $\mathbb{M}$. If
\begin{enumerate}[label={\textbf{(\roman{*})}}]
\item $G_1=N_1+M_1$ and
\item $N_1\cap M_1=0$
\end{enumerate}
then there exist a morphism of group-groupoids $\tau\colon\mathbb{M}\rightarrow \A(\N)$, i.e. an action of $\mathbb{M}$ on $\N$ defined as follows:
\[m \cdot n  = m + n - m \]
for all $m\in M_1$ and $n\in N_1$. Then the semi-direct product group-groupoid $\N\rtimes_{\tau}\mathbb{M}$ obtained from the action $\tau$ is isomorphic to the group-groupoid $\G$.

The notion of crossed modules over group-groupoids is introduced in \cite{Temel2018} using split extensions of group-groupoids by the method used in \cite{Porter1987}. We recall the definition from \cite{Temel2018}.

\begin{definition}\cite{Temel2018}
	A crossed module over group-groupoids consist of group-groupoid morphisms $\alpha\colon\G\rightarrow\H$ and $\theta\colon\H\rightarrow\A(\G)$, i.e. a group-groupoid action of $\H$ on $\G$, such that $\alpha_1\colon G_1\rightarrow H_1$ is a crossed module over groups.
\end{definition}

\[\xymatrix@R=10mm@C=10mm{
	G_1 \ar[r]^{\alpha_1} \ar@<.3ex>[d]^-{d_0} \ar@<-.3ex>[d]_-{d_1} 
	& H1 \ar@<.3ex>[d]^-{d_0} \ar@<-.3ex>[d]_-{d_1}   
	\\	
	G_0 \ar@/^/@<2ex>[u]^-{\varepsilon} \ar[r]_{\alpha_0} 
	& H_0 \ar@/_/@<-2ex>[u]_-{\varepsilon}  
}\]

We know that for a group $A$, the inner automorphism map $A\rightarrow\Aut(A)$ forms a crossed module structure over groups. Analogously we can obtain this fact for group-groupoids.

\begin{proposition}
Let $\G$ be a group-groupoid. Then the canonical group-groupoid morphism $\varphi\colon\G\rightarrow \A(\G)$ is equipped with the structure of crossed module over group-groupoids.
\end{proposition}

\begin{proof}
The proof is straightforward and is omitted.
\end{proof}

\subsection{Holomorph of a group-groupoid}

A concept in group theory which arose in connection with the following problem: "Is it possible to include any given group $A$ as a normal subgroup in some other group so that all the automorphisms of $A$ are restrictions of inner automorphisms of this large group?". This problem has been answered positively and this large group is called the holomorph of $A$ \cite{Hall1976,Kurosh1960}.

The holomorph of a group is a group which simultaneously contains copies of the group and its automorphism group. The holomorph provides interesting examples of groups, and allows one to treat group elements and group automorphisms in a uniform context. In group theory, for a group $A$, the holomorph of $A$ denoted $\Hol(A)$ can be described as a semi-direct product group $A\rtimes\Aut(A)$ or as a permutation group.

The problem mentioned above has also positive answer for group-groupoids. Statement of the problem for group-groupoids is as follows: "Is it possible to include any given group-groupoid $\G$ as a normal subgroup-groupoid in some other group-groupoid so that all the actors of $\G$ are restrictions of inner actors of this large group-groupoid?". This large group-groupoid is the holomorph of $\G$.

\begin{definition}
The \textbf{holomorph} $\Hol(\G)$ of a group-groupoid $\G$ is the semi-direct product group-groupoid $\G\rtimes\A(\G)$ where the action $\theta\colon\A(\G)\rightarrow \A(\G)$ is the conjugation.
\end{definition}

\[\Hol(\G)=\G\rtimes\A(\G)=(G_1\rtimes\W(\G),G_0\rtimes\Aut(\G))\]

\begin{definition}
A subgroup-groupoid $\H$ of a group-groupoid $\G$ is called \textbf{characteristic} in $\G$ if restriction defines a group-groupoid morphism $\A(\G)\rightarrow\A(\H)$.
\end{definition}

Characteristic subgroup-groupoids are subgroup-groupoids which are invariant under all inner actors. In \cite{Scott1964} it has been shown that there is a bijection between the set of characteristic subgroups of a group $A$ and the set of all normal subgroups of $\Hol(A)$ contained in $A$. Now we give a similar result for group-groupoids.

\begin{lemma}
A subgroup-groupoid of $\G$ is characteristic if and only if its image in $\Hol(\G)$ is a normal subgroup-groupoid.
\end{lemma}

\begin{example}
Let $\G$ be a group-groupoid. Then the derived subgroup-groupoid $\G'=[\G,\G]$ is characteristic. 
\end{example}

\bibliography{Sahan}
\bibliographystyle{abbrv}

\end{document}